\documentclass[a4paper,10pt]{article}
\usepackage[utf8]{inputenc}
\usepackage{amssymb, amsmath, amsthm}
\usepackage{array}
\usepackage{textcomp}
\usepackage{color}
\usepackage{stmaryrd}
\usepackage{framed}
\usepackage[utf8]{inputenc}
\usepackage{default}
\usepackage{amsmath}
\usepackage{amsfonts}
\usepackage{amssymb}
\usepackage[english]{babel}
\usepackage{graphicx}
\usepackage{authblk}

\usepackage{amsmath,amssymb,amsthm}
\usepackage{graphicx,color}
\usepackage{hyperref}
\usepackage{float}
\usepackage{capt-of}
\usepackage{etex}
\usepackage{mathtools}
\usepackage{stmaryrd}
\usepackage{authblk}

\newtheorem{theorem}{Theorem}
\newtheorem{lemma}{Lemma}

\newtheorem{remark}{Remark}

\title{A Local Mesh Modification Strategy for Interface Problems with Application to Shape and Topology Optimization}
\author[1,2]{ P. Gangl \thanks{peter.gangl@dk-compmath.jku.at}}
\author[3]{U. Langer \thanks{ulanger@numa.uni-linz.ac.at}}
\affil[1]{Doctoral Program ``Comp. Mathematics'', JKU Linz, Austria}
\affil[2]{Linz Center of Mechatronics GmbH (LCM), Linz, Austria\\A}
\affil[3]{Institute of Computational Mathematics, JKU Linz, Austria\\A}

\begin{document}

\maketitle

\begin{abstract}
We present and analyze a new finite element method for solving interface problems on a triangular grid. The method locally modifies a given triangulation such that the interfaces are accurately resolved and the maximal angle condition holds. Therefore, optimal order of convergence can be shown. Moreover, an appropriate scaling of the basis functions yields an optimal condition number of the stiffness matrix. The method is applied to an optimal design problem for an electric motor where the interface between different materials is evolving in the course of the optimization procedure.
\end{abstract}
\section{Motivation}\label{gangl:sec_motivation}
Our research is motivated by the design optimization of an electric motor by means of topology and shape optimization.
We are interested in finding the optimal distribution of two materials (usually ferromagnetic material and air) within a fixed design subdomain of an electric motor, see, e.g. \cite{GanglLangerLaurainMeftahiSturm2015}. 
We use a two-dimensional model for the electric motor, 
which is widely used for this kind of applications.
In the optimization procedure, one usually starts with an initial guess and then uses shape sensitivities or topological sensitivities
to gradually improve the initial design. In the course of this optimization procedure, the interface between the two subdomains evolves. For computing the sensitivities that steer the optimization process, it is necessary to solve the state equation and the adjoint equation in each optimization iteration, which is usually done by the finite element method. 
Besides remeshing in every iteration, which is very costly, and advecting the whole mesh in every step of the optimization procedure, which may cause self-intersection of the mesh, there exist several other methods in the literature which can deal with these kinds of interface problems.
We mention the XFEM, 
which uses local enrichment of the finite element basis, and the unfitted Nitsche method.
In \cite{FreiRichter2014}, the authors introduce a locally modified parametric finite element method based on a quadrilateral mesh with a patch structure. We present an adaptation of this method to the case of finite elements on triangular meshes. 
One advantage of this kind of method over the ones mentioned before is that this method has a fixed number of unknowns independently of the position of the interface relative to the mesh. The given mesh is modified only locally near the material interface. The method is relatively easy to implement and we can show optimal order of convergence.
\section{A local mesh modification strategy for Interface Problems}\label{gangl:sec_method}
We introduce the method for the 
potential
equation 
in
a bounded, polygonal computational domain $\Omega \subset \mathbb R^2$ consisting of two non-overlapping subdomains, $\overline \Omega = \overline \Omega_1 \cup \overline \Omega_2$, $\Omega_1 \cap \Omega_2 = \emptyset$, which represent two materials with different material coefficients $\kappa_1, \kappa_2>0$. On the material interface $\Gamma := \overline \Omega_1 \cap \overline \Omega_2$, 
we have to require that the solution as well as the flux are continuous.
For simplicity,
we assume homogeneous Dirichlet boundary conditions on $\partial \Omega$. 
The problem reads as follows:
\begin{align}	\label{gangl:interfaceProblem}
	\begin{aligned}
		-\mbox{div }\left(\kappa_i \nabla u\right) &= f \quad \mbox{ in } \Omega_i, \;\; i =1,2, \\
		\left[u \right] &= 0 \quad \mbox{ on } \Gamma, \\
		\left[\kappa \frac{\partial u}{\partial n}\right] &= 0 \quad \mbox{ on } \Gamma, \\		
		u &= 0 \quad \mbox{ on } \partial \Omega,
	\end{aligned}	
\end{align}
where we assume that the boundaries of the two subdomains as well as the right hand side $f$ are sufficiently regular such that
$
	u \in H_0^1(\Omega)\cap H^2(\Omega_1\cup\Omega_2), 
$
that means that the restrictions of $u \in H_0^1(\Omega)$ to $\Omega_1$ and $\Omega_2$ 
belong to $H^2(\Omega_1)$ and $H^2(\Omega_2)$, respectively, see, e.g., \cite{Babuska1970}. It is well-known that, when using standard finite element methods, the interface must be resolved by the mesh in order to obtain optimal convergence rates of the approximate solution $u_h$ to the true solution $u$ in the $L^2$ and $H^1$ norms as the mesh parameter $h$ 
tends
to zero, see also \cite{FreiRichter2014}.
The discretization error estimate is usually shown using an interpolation error estimate.
A condition that is sufficient and necessary for such an interpolation error estimate is that all interior angles of triangles of the mesh are bounded away from 180$^\circ$ (maximum angle condition), see \cite{Babuska1976}.
\subsection{Preliminaries}
Let $\mathcal T_h$ be a shape-regular and quasi-uniform subdivision of $\Omega$ into triangular elements,
and let us denote 
the space of globally continuous, piecewise linear functions on $\mathcal T_h$ 
by $V_h$ . 
We assume that $\mathcal T_h$ has been obtained by one uniform refinement of a coarser mesh $\mathcal T_{2h}$. By this assumption, $\mathcal T_h$ has a patch-hierarchy, i.e., always four elements $T_1$, $T_2$, $T_3$, $T_4 \in \mathcal T_h$ can be combined to one larger triangle $T \in \mathcal T_{2h}$. 
We will refer to this larger element as the makro element or patch.
 We assume further that the mesh of makro elements $\mathcal T_{2h}$ is such that, for each makro element $T$, the interface $\Gamma$ either does not intersect the interior of $T$, or such that $\Gamma$ intersects $T$ in exactly two distinct edges or that it intersects $T$ in one vertex and in the opposite edge. For a smooth enough interface $\Gamma$, this assumption can 
always be enforced by choosing a fine enough makro mesh $\mathcal T_{2h}$.
We consider a makro element $T\in\mathcal T_{2h}$ to be cut by the interface if the intersection of the interior of makro element with interface is not the empty set.
\subsection{Description of the method}
The method presented in this paper is a local mesh adaptation strategy, meaning that only makro elements close to the interface $\Gamma$ will be modified. 
Given the hierarchic structure of the mesh, on every makro element we have four elements of the mesh $\mathcal T_h$ and six vertices, see Fig. \ref{gangl:fig_patch}(a),(b).
The idea of the method is the following: For each makro element that is cut by the interface, move the points $P_4$, $P_5$ and $P_6$ along the corresponding edges in such a way that, on the one hand, the interface is resolved accurately, and, on the other hand, all interior angles in the four triangles are bounded away from 180$^\circ$. 
For a makro element $T$ that is cut by the interface, we distinguish four different configurations as follows:

In the case where the makro element is cut by the interface in two distinct edges, we denote the vertex of the makro element where these two edges meet by $P_1$, and the other two vertices in counter-clockwise order by $P_2$ and $P_3$. The parameters $s$, $t$, $r \in [0,1]$ represent the positions of the points $P_4$, $P_5$, $P_6$ along the corresponding edges by
\begin{align*}
	P_4(s) = P_1 + s\, \frac{P_2 - P_1}{|P_2-P_1|}, \quad
	P_5(t) = P_2 + t\, \frac{P_3 - P_2}{|P_3-P_2|}, \quad
	P_6(r) = P_1 + r\, \frac{P_3 - P_1}{|P_3-P_1|}. 
\end{align*}
The parameters $r$ and $s$ will always be chosen in such a way that the intersection points of the interface and the edges $P_1P_3$ and $P_1P_2$ are the points $P_6$ and $P_4$, respectively. Thus, we identify the position of the interface relative to the makro element $T$ by the two parameters $r,s$. 
We choose the parameter $t$ such that a maximal angle condition is satisfied as follows:\\
\textbf{Configuration A:} $0<r, s\leq1/2$. Set $t=1/2$.\\
\textbf{Configuration B:} $1/2 <r, s<1$. Set $t=1-s$.\\
\textbf{Configuration C:} $0 < s \leq 1/2<r<1$ or $0 < r \leq 1/2<s<1$. Set $t=1/2$.

The case where the makro element is cut in one vertex and the opposite edge has to be considered separately. We denote the vertex of the makro element where it is cut by the interface by $P_2$ and the other vertices, in counter-clockwise ordering, by $P_3$ and $P_1$, see Fig. \ref{gangl:fig_patch}(b). The location of the interface is given by the position of the point $P_6$ on the edge between $P_3$ and $P_1$. In this case, we also need to rearrange the triangles $T_2$ and $T_4$.\\
\textbf{Configuration D:} \\
\textit{Configuration D1:} $ 0< r \leq 1/2$. Set $s = r$ and $t=1/2$.\\
\textit{Configuration D2:} $ 1/2 < r < 1$. Set $s = 1/2$ and $t=r$.

With this setting, it is possible to show the required maximal angle condition on the reference patch $\hat T$ defined by the outer makro vertices $\hat P_1 =(0,0)^T$, $\hat P_2 = (1,0)^T$, $\hat P_3=(1/2, \sqrt{3}/2)^T$. 
\begin{figure}
	\centering
		\begin{tabular}{ccc}
			\includegraphics[scale=0.35]{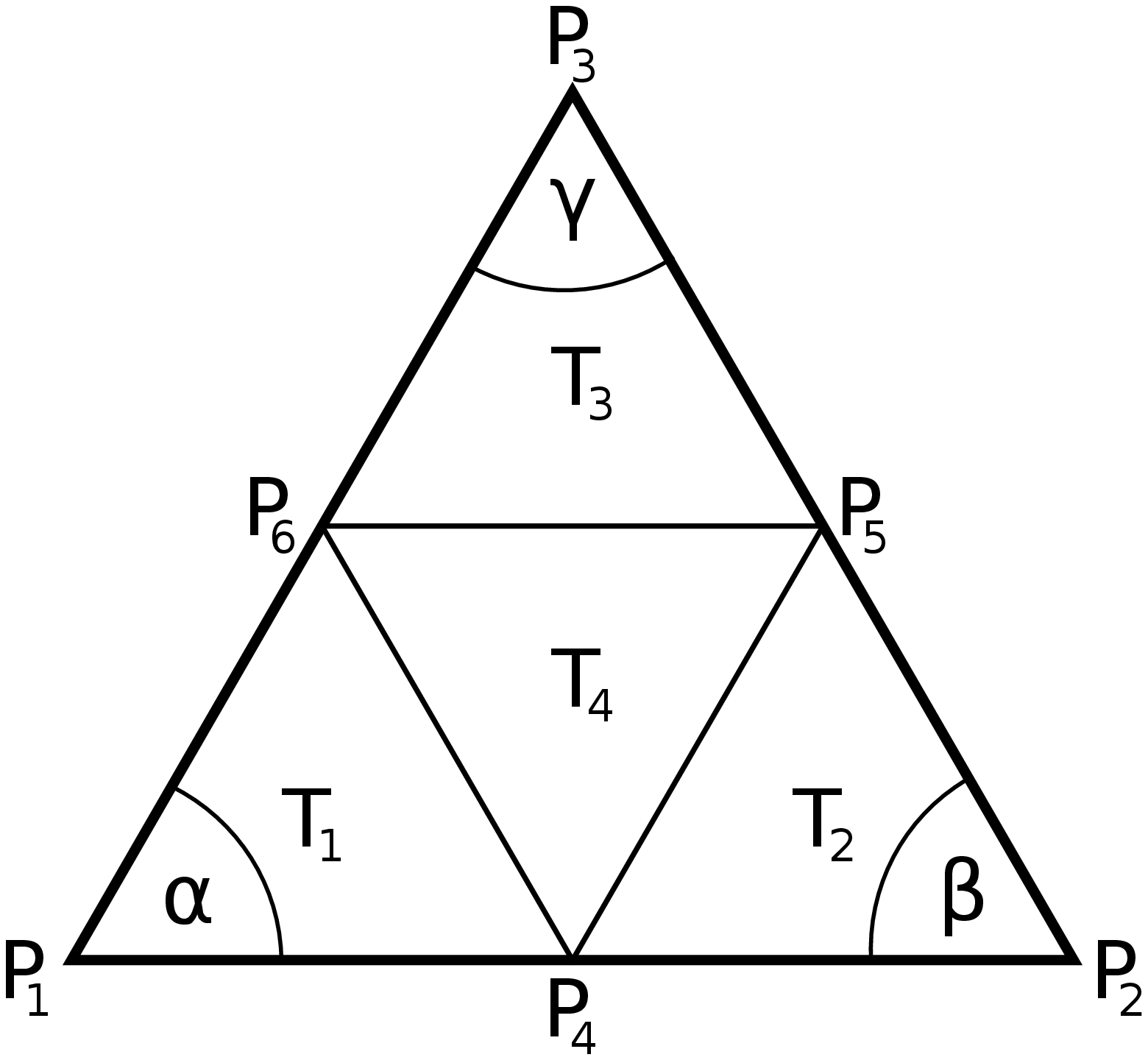} & \includegraphics[scale=0.35]{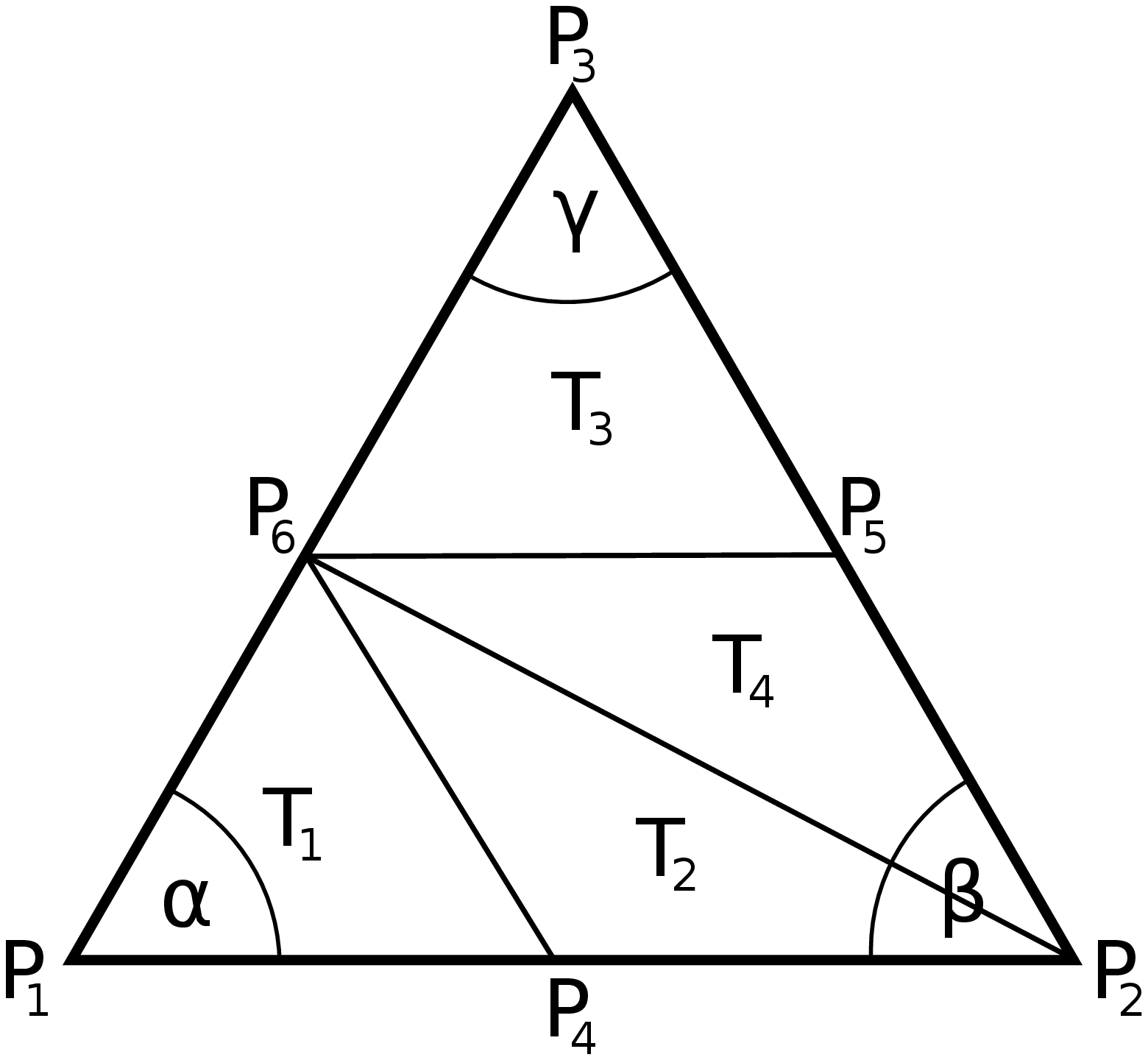} \\
			(a) Patch for configurations A--C & (b) Patch for configuration D \\
			 \includegraphics[scale=0.35]{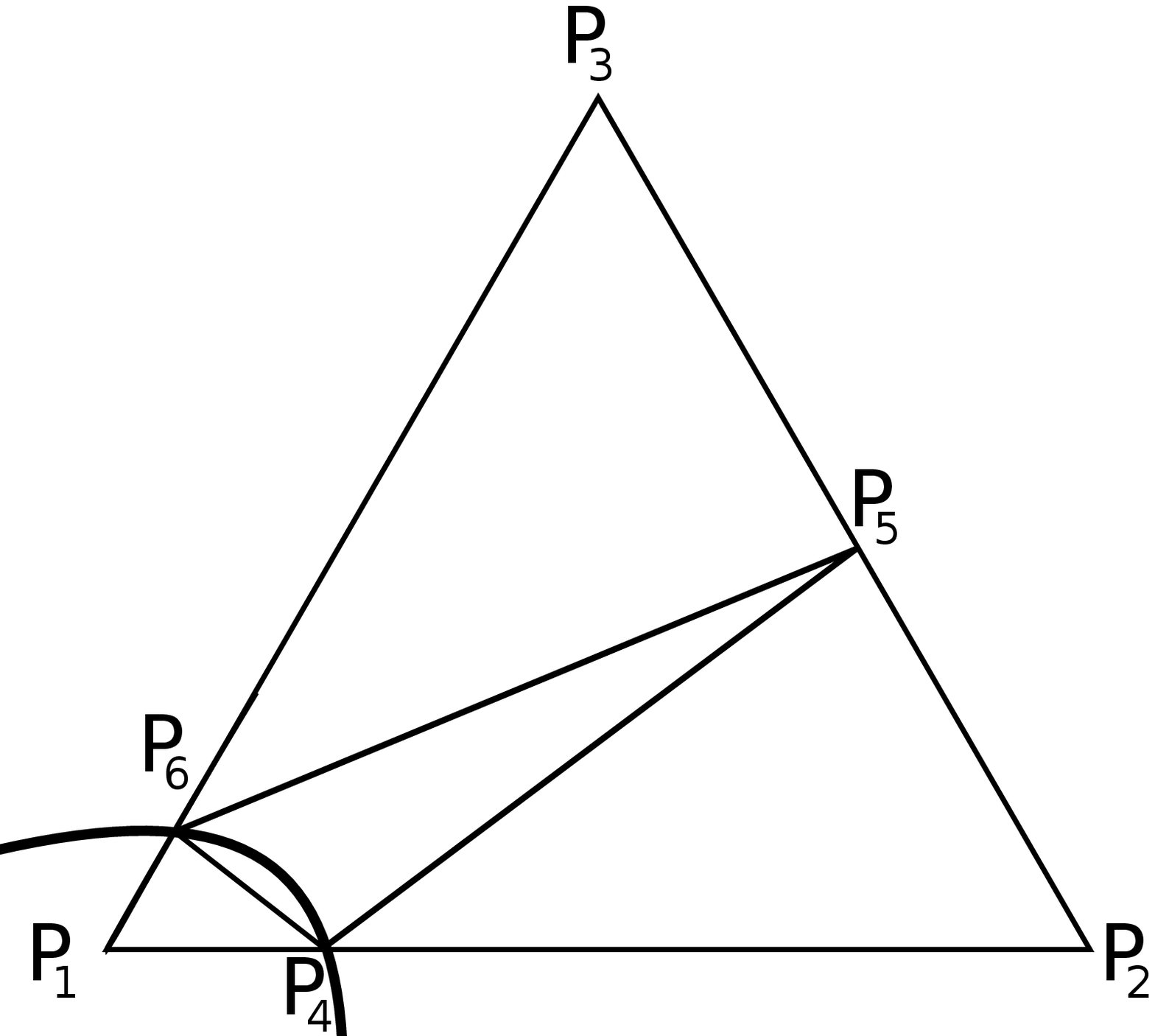} & \includegraphics[scale=0.35]{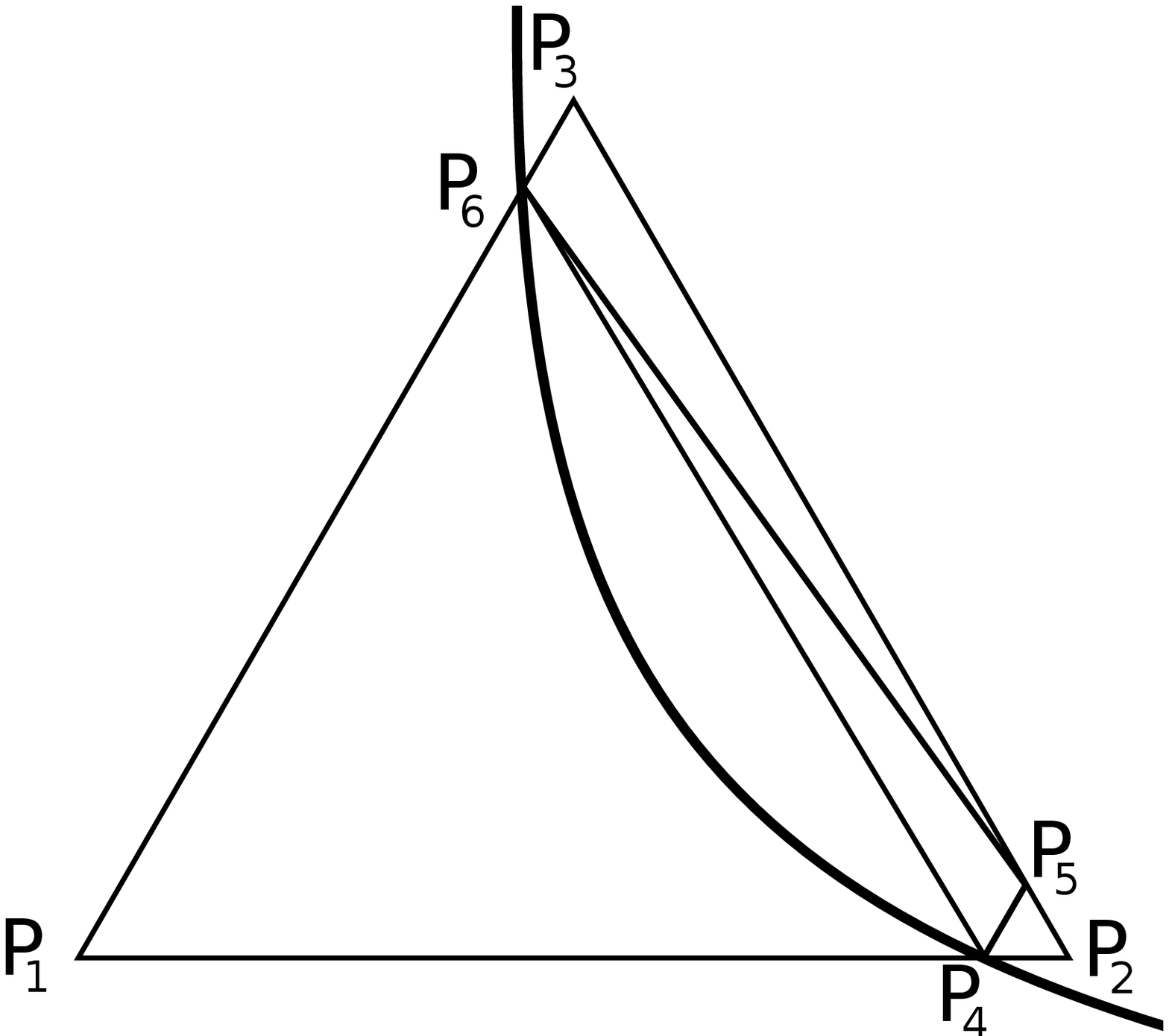}\\
			  (c) Configuration A & (d) Configuration B \\
			   \includegraphics[scale=0.35]{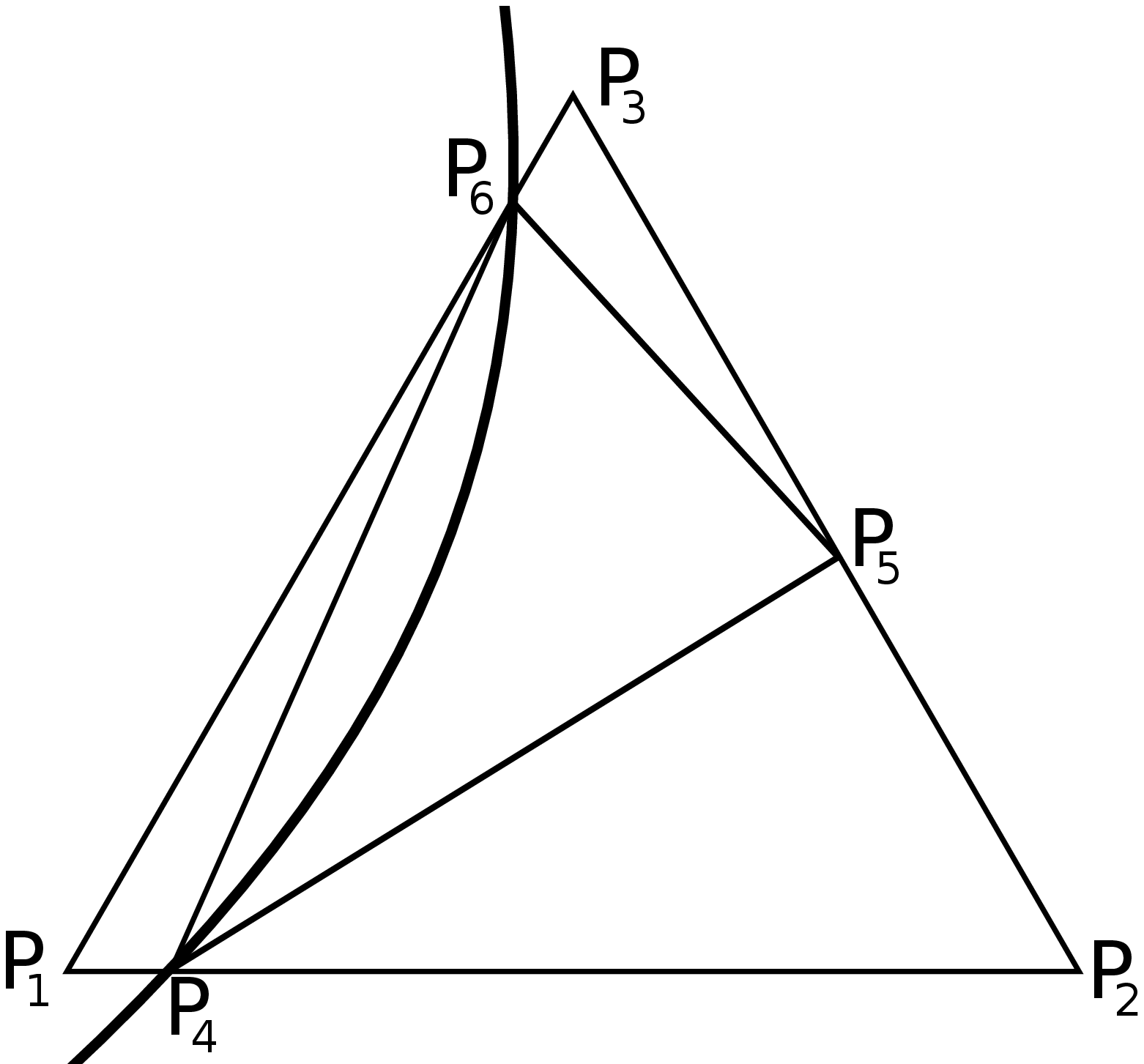} & \includegraphics[scale=0.35]{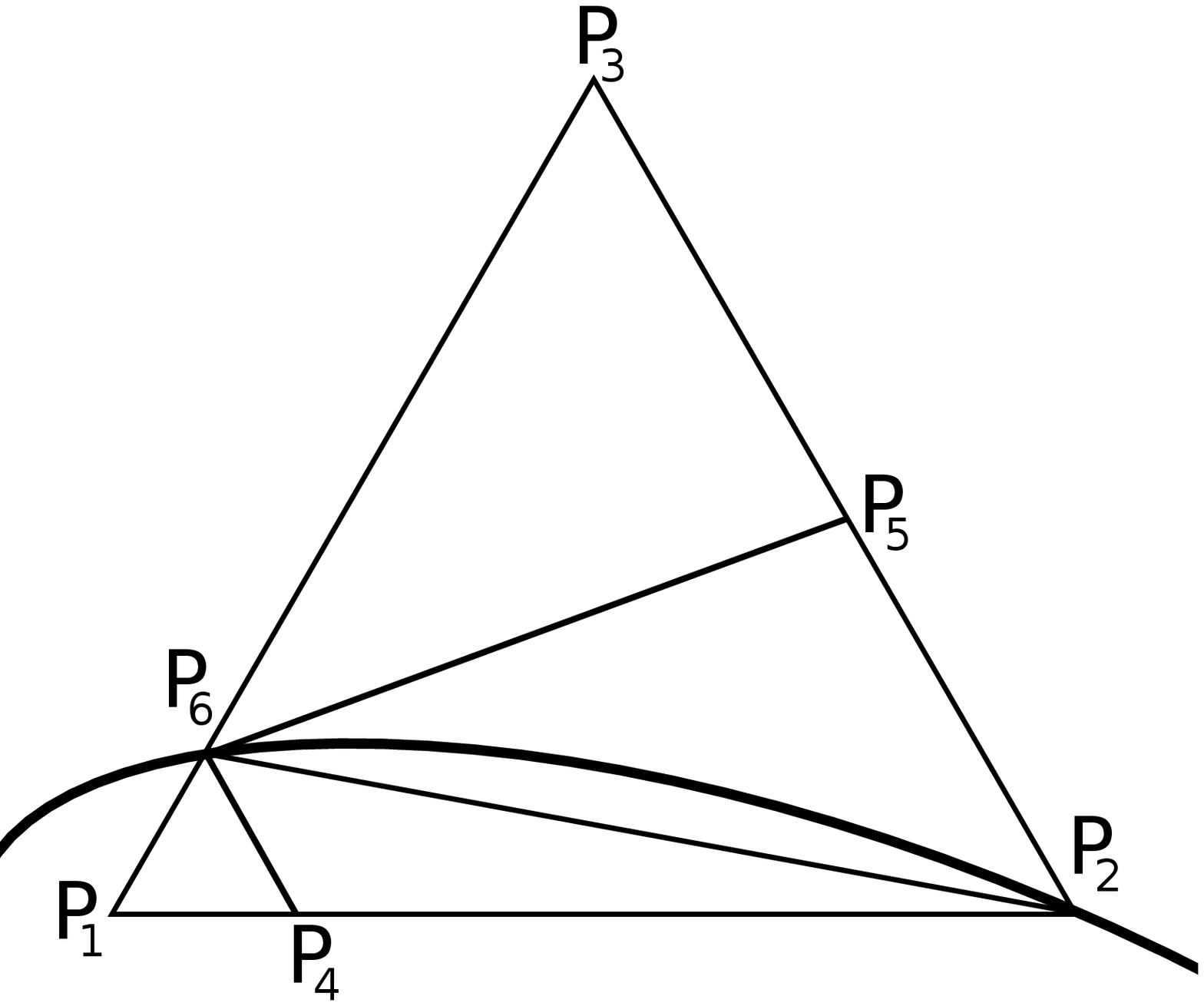} \\
			   (e) Configuration C & (f) Configuration D
			\end{tabular}
	\caption{(a), (b): Patches for different configurations. (c)-(f): Different configurations of mesh points depending on position of the interface.}
	\label{gangl:fig_patch}
\end{figure}
\begin{lemma} \label{gangl:lem_MaxAngle}
	All angles in triangles of the reference patch $\hat T$ are bounded by $150^\circ$ independent of the parameters $r,s \in [0,1]$.
\end{lemma}
\begin{proof}
	We have to ensure for each of the four subtriangles $\hat T_1$, $\hat T_2$, $\hat T_3$, $\hat T_4$ that all of their three interior angles are not larger than 150$^\circ$. 
	In Configuration A--C, the sub-triangles $\hat T_1$, $\hat T_2$ and $\hat T_3$ all have one angle of 60$^\circ$. Obviously, the remaining two angles are bounded from above by 120$^\circ$. The same holds true for the sub-triangles $\hat T_1$ and $\hat T_3$ in Configuration D.
	
	For three points $A$, $B$, $C$ in $\mathbb R^2$, define
	\begin{align*}
		\measuredangle(A,B,C) := \mbox{cos}^{-1}\left( \frac{\left( A-B, C-B \right)}{|A-B|\, |C-B|} \right)
	\end{align*}
	the interior angle of the triangle with vertices $A$, $B$, $C$ at point $B$.
	
	\textit{Configuration A:} For $r,s \in (0, 1/2]$, we get for the angle in point $P_4$ that
	\begin{align*}
		\measuredangle(P_6,P_4,P_5) < \measuredangle(P_1,P_4,P_5) = 180^\circ - \measuredangle(P_5,P_4,P_2) \leq  180^\circ - \measuredangle(P_5,P_1,P_2).
	\end{align*}
	Since the reference patch $\hat T$ is equilateral, it holds $\measuredangle(P_5,P_1,P_2) = \alpha/2$.
	Analogously, we get for the angle in point $P_6$ that $\measuredangle(P_5,P_6,P_4) < 180^\circ - \alpha /2$.
	It is easy to see that the angle in point $P_5$ increases with $r,s$ and thus is maximized for $r=s= 1/2$, which yields that 
	$
		\measuredangle(P_4,P_5,P_6) \leq \measuredangle(P_4(1/2), P_5, P_6(1/2) ) = 180- \beta - \gamma = \alpha.
	$
	Here we used that, for $r=s=t=1/2$, the four sub-triangles are congruent.
	
	\textit{Configuration B:}
	Note that, by the special choice of $s, t$, in this case we have that the line going through $P_4$ and $P_5$ is parallel to the edge connecting $P_1$ and $P_3$ for all values of $s\in(1/2,1)$. Thus, we have
	\begin{align*}
		\measuredangle(P_4,P_5,P_6) &\leq \measuredangle(P_4,P_5,P_3) = 180^\circ - \gamma \quad \mbox{ and } \\
		\measuredangle(P_4,P_5,P_6) &= 180^\circ - \gamma - \measuredangle(P_6, P_5, P_3) \\
			&\geq 180^\circ - \gamma - \measuredangle( P_6(1/2), P_5(1/2), P_3) = 180^\circ - \gamma - \beta = \alpha .
	\end{align*}
	The angles in $P_4$ and in $P_6$ must also be bounded from above by $180^\circ - \alpha = 120^ \circ$.
	
	\textit{Configuration C:} We consider the case where $r \in (1/2,1)$ and $s\in (0,1/2]$. The reverse case is treated analogously. For the angle in the fixed point $P_5 = P_5(1/2) = (P_2+P_3)/2$, we get the estimates
	\begin{align*}
		\measuredangle(P_4,P_5,P_6) &\leq \measuredangle(P_4,P_5,P_3) \leq \measuredangle( (P_4(1/2),P_5,P_3) = 180^\circ - \gamma,\\
		\measuredangle(P_4,P_5,P_6) &\geq \measuredangle(P_4,P_5,P_6(1/2) ) \geq \measuredangle( P_1,P_5, P_6(1/2)) = \measuredangle( P_5,P_1, P_2) = \alpha/2.
	\end{align*}
	Thus, the angles $\measuredangle(P_6, P_4,P_5)$ and $\measuredangle(P_5,P_6,P_4)$ are also bounded from above by $180^\circ -\beta/2$.

	\textit{Configuration D:} We consider only Configuration D1, the corresponding result for Configuration D2 follows analogously.
	Due to the choice of the parameter $s$, the line going through $P_4$ and $P_6$ is parallel to the edge connecting $P_2$ and $P_3$ for all values of $r$. We need to consider triangles $T_2$ and $T_4$. In $T_2$, 
	$
		\measuredangle(P_6, P_4, P_2) = 180^\circ - \beta
	$
	and, therefore, the other two angles are bounded by $\beta$. In $T_4$, we have for $r \in (0,1/2]$ that
	\begin{align*}
		\measuredangle(P_6, P_2, P_5) &\leq \beta,\\
		\measuredangle(P_2, P_5, P_6) &\leq \measuredangle(P_2, P_5, (P_3+P_1)/2) = 180 - \beta,\\
		\measuredangle(P_5, P_6, P_2) &\leq \measuredangle(P_3, P_6, P_4) = 180 - \gamma.
	\end{align*}
	Finally, noting that $\alpha = \beta = \gamma = 60^\circ$ yields the statement of the lemma.
\end{proof}
\begin{remark}
	Due to the assumption that the makro mesh is shape-regular, we 
	obtain a maximal angle condition (with a different bound) for all triangles of the mesh $\mathcal T_h$.
\end{remark}
Now we are in the position to show an a priori error estimate for the finite element solution $u_h$.
Since we have the maximum angle condition of Lemma \ref{gangl:lem_MaxAngle}, 
we get the interpolation error estimates
\begin{align}	\label{gangl:interpolEst}
	\| \nabla^k(v - I_h v) \|_{L^2(T)} \leq c\, h^{2-k}_{T,max} \| \nabla^2 v \|_T, \quad k=0,1,
\end{align}
where 
$I_h: H^2(T) \rightarrow V_h|_{\overline T}$ denotes the Lagrangian interpolation operator,
$c$ is a positive generic constant,
and $h_{T,max}$ is the maximum edge length of the triangle $T \in \mathcal T_h$, see, e.g., \cite{Apel1999}.
In the case where the interface $\Gamma$ is not polygonal but smooth with $C^2$ parametrization, and an element of the mesh $\mathcal T_h$ is intersected by $\Gamma$, the solution $u$ is not smooth across the interface and, hence, estimate \eqref{gangl:interpolEst} cannot be applied. However, the same estimate with $k=1$ was shown in \cite{Frei2016}. These interpolation error estimates allow to show the following a priori error estimate \cite{Frei2016}.
\begin{theorem} \label{gangl:theoErrEst}
	Let $\Omega \subset \mathbb R^2$ be a domain with convex polygonal boundary, split into $\Omega = \Omega_1 \cup \Gamma \cup \Omega_2$, where $\Gamma$ is a smooth interface with $C^2$-parametrization. We assume that $\Gamma$ divides $\Omega$ in such a way that the solution 
	$u$ belongs to 	 $H_0^1(\Omega)\cap H^2(\Omega_1 \cup \Omega_2)$ and
	satisfies the stability estimate 
	$\|u\|_{H^2(\Omega_1 \cup \Omega_2)} \leq c_s \|f\|$.	
	Then, for the corresponding modified finite element solution $u_h\in V_h$, we have
	the
	estimates
	\begin{align*}
		\| \nabla (u-u_h)\|_{L^2(\Omega)} \leq C\,  h \,  \|f\| \;\; \mbox{and} \;\; \|u-u_h\|_{L^2(\Omega)} \leq C \, h^2 \, \|f\|.
	\end{align*}
	\end{theorem}
\section{Condition number}
The procedure of Section \ref{gangl:sec_method} guarantees that no angle of the modified mesh becomes too large. However, it may happen that some angles in the triangulation are getting arbitrarily close to zero, which usually yields a bad condition of the finite element system matrix. 
This problem was also addressed in \cite{FreiRichter2014} for the case of quadrilateral elements, and we can adapt the procedure to the triangular case.

The idea consists in a hierarchical splitting of the finite element space $V_h = V_{2h} + V_b$ into the standard piecewise linear finite element space on the makro mesh $\mathcal T_{2h}$ and the space of ``bubble'' functions in $V_b$ which vanish on the nodes of the makro elements. 
Let $\lbrace \phi_h^1, \dots, \phi_h^{N_h}\rbrace$ be the nodal basis of the space $V_h$.
Any function $v_h \in V_h$ can be decomposed into the sum of a function $v_{2h} \in V_{2h} = \mbox{span}\lbrace \phi_{2h}^1, \dots, \phi_{2h}^{N_{2h}} \rbrace$ and a function $v_b \in V_b=\lbrace \phi_{b}^1, \dots, \phi_{b}^{N_{b}} \rbrace$,
\begin{align*}
	v_h = \sum_{i=1}^{N_h} \mathbf v_h^i \phi_h^i = \sum_{i=1}^{N_{2h}} \mathbf v_{2h}^i \phi_{2h}^i + \sum_{i=1}^{N_b} \mathbf v_b^i \phi_b^i = v_{2h}+v_b \in V_{2h}+V_b.
\end{align*}
In this setting it is possible to scale the basis functions $\phi_b^i$ of the space $V_b$ in such a way that the following two conditions are satisfied:
\begin{itemize}
	\item There exists a constant $C>0$ independent of $h$, $r$, $s$ such that
	\begin{align}
		C^{-1} \leq \| \nabla \phi_h^i \| \leq C, \quad i=1,\dots,N_h \label{gangl:condition1},
	\end{align}
	\item There exists a constant $C>0$ independent of $h$, $r$, $s$, such that for all $v_b \in V_b$
	\begin{align}
		|\mathbf v_b^i| \leq C \| v_b \|_{\mathcal N_i}, \quad i = 1,\dots,N_b, \label{gangl:condition2}
	\end{align}
	where $\mathcal N_I = \lbrace K \in \mathcal T_h: x_i \in \overline K\rbrace$.
\end{itemize}
Under these two assumptions it is possible to show the usual 
bound on the condition number of the system matrix:
\begin{theorem}
	Assume that \eqref{gangl:condition1} and \eqref{gangl:condition2} hold. Then there exists a constant $C>0$ independent of $r$, $s$, such that
	$\mbox{cond}_2(\mathbf A) \leq C \, h^{-2}$.
\end{theorem}
\section{Numerical Results}
We implemented the method described in Section \ref{gangl:sec_method}, and tested it for the example where $\Omega = (0,1)^2$, $\Omega_1= B(0,1/2)$, $\Omega_2 = \Omega \setminus \overline \Omega_1$, $\kappa_1=1$, $\kappa_2=10$, and the right hand side as well as the Dirichlet data were chosen in such a way that the exact solution is 
known explicitly.
The optimal order of convergence stated in Theorem \ref{gangl:theoErrEst} can be observed in Table \ref{gangl:tableNumRes1}.
The interface method was also included in the shape optimization of an electric motor described in \cite{GanglLangerLaurainMeftahiSturm2015}. It can be seen from Fig. \ref{gangl:fig_motorOpt} that smoother and better designs can be achieved by locally modifying the mesh nodes.
\begin{table}
	\centering
	\label{gangl:tableNumRes1}
	\begin{tabular}{r  p{8mm}p{2cm}p{11mm}p{23mm}p{12mm}p{1cm}}
		\hline
		nVerts	& 		h			&  $\| u-u_h\|_{L^2}$	& rate $L_2$ 	 	&  $\| \nabla(u-u_h)\|_{L^2}$	& rate $H_1$ & angMax	 \\ 
		\hline
		289		&	$h_0$			& 	0.00724623		&		--	 			& 	0.175665					&	--	& 140.334 \\  
		1089	&	$h_0/2$			& 	0.00180955		&	\textbf{2.0016}	  	& 	0.087845				&	\textbf{0.9998}	  &138.116\\  
		4225	&	$h_0/4$			& 	0.000453133			&	\textbf{1.9976}		& 	0.0439104				&	\textbf{1.0004}	& 143.084\\  
		16641	&	$h_0/8$			& 	0.000113451			&	\textbf{1.9979}	 	& 	 0.0219536				&	\textbf{1.0001}	&152.223 \\  
		66049	&	$h_0/16$		& 	0.0000283643		&	\textbf{1.9999}		& 	0.0109756				&	\textbf{1.0002} &149.110	\\  
		263169	&	$h_0/32$		& 	0.00000709548	&	\textbf{1.9991}		& 	0.00548762				&	\textbf{1.0001}	 &155.643\\  
		\hline
	\end{tabular}
	\caption{Convergence history of interface problem \eqref{gangl:interfaceProblem} using mesh adaptation strategy}
\end{table}
\begin{figure}[ht]
	\centering
		\begin{tabular}{ccc}
			\includegraphics[scale=0.35, trim= 0mm 0mm 0mm 0mm, clip=true]{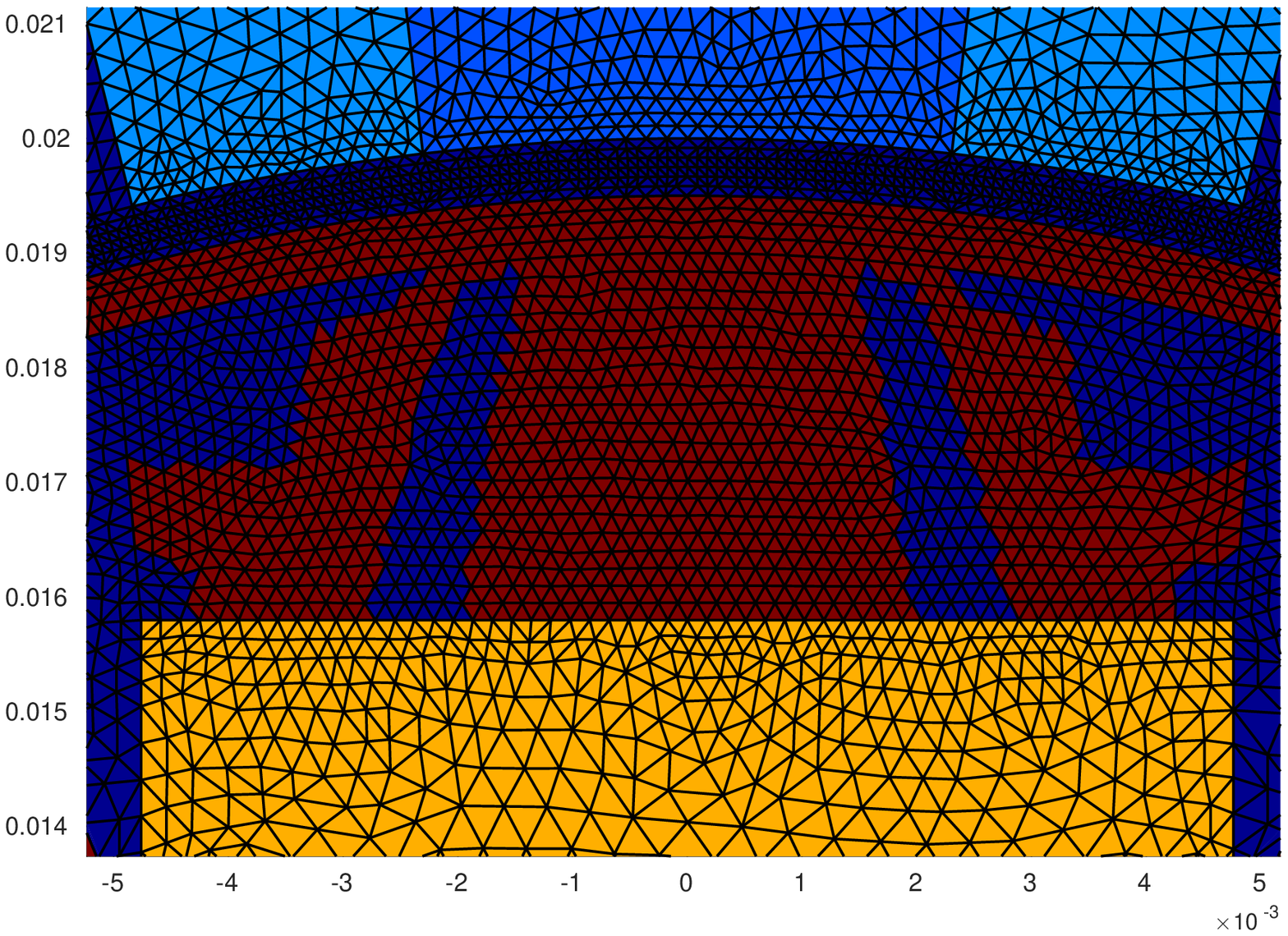} & \includegraphics[scale=0.19, trim= 0 00mm 0 0mm, clip=true]{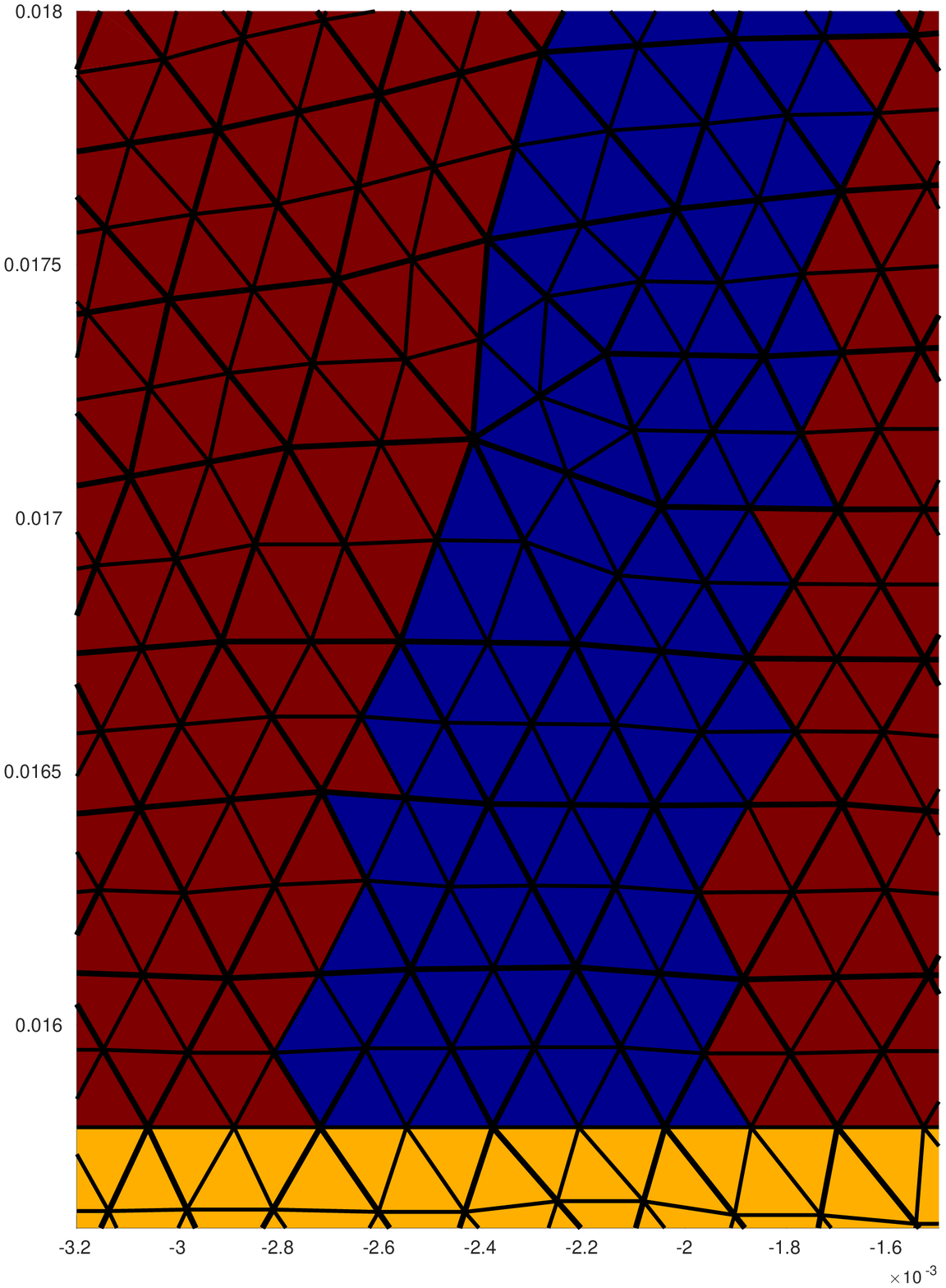} \\
			(a) & (b)\\
			\includegraphics[scale=0.35, trim= 0mm 0mm 0mm 0mm, clip=true]{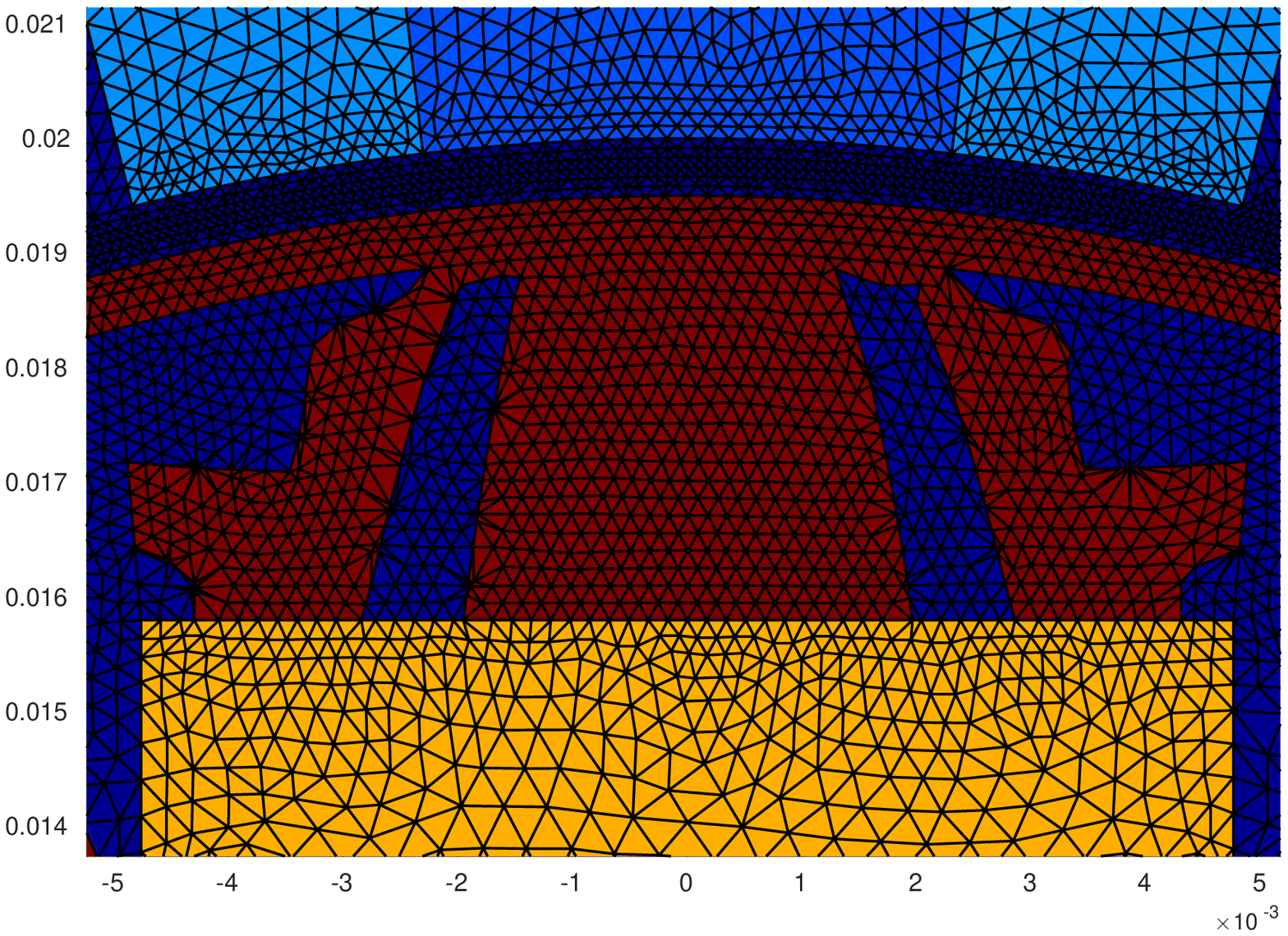} & \includegraphics[scale=0.19, trim= 0 0mm 0 0mm, clip=true]{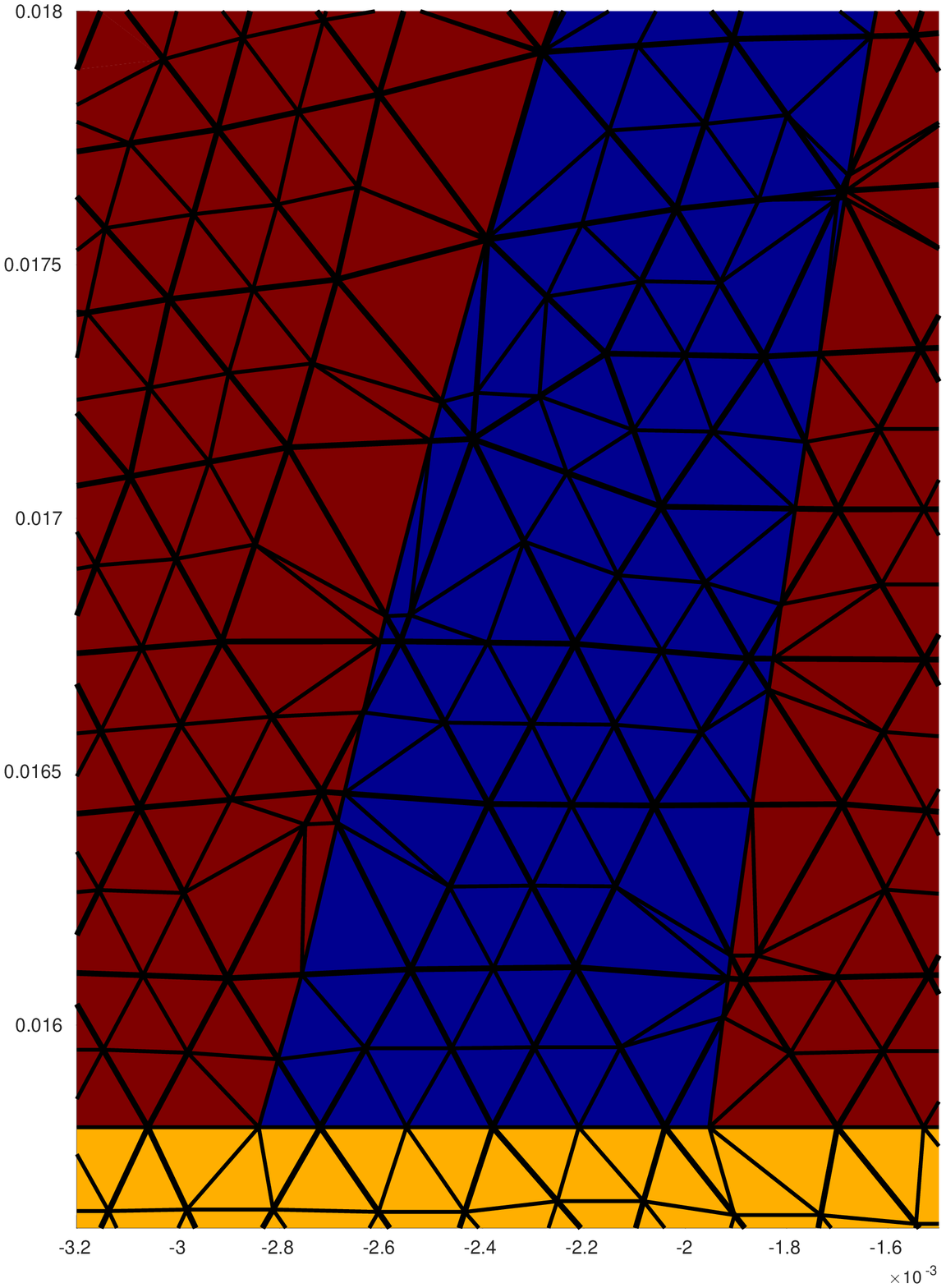}\\
			 (c) & (d)
			\end{tabular}
	\caption{(a) Final design of shape optimization without interface method, objective value $\mathcal J(u) \approx 0.0379$. (b) Zoom of (a). (c) Final design of shape optimization with interface method, objective value $\mathcal J(u) \approx  0.0373$. (d) Zoom of (c).}
 	\label{gangl:fig_motorOpt}
\end{figure}
\section{Conclusion}
We presented a local mesh modification strategy which allows to accurately resolve interfaces using the finite element method. We showed a maximal angle condition which ensures optimal order of convergence and presented numerical results for a model problem and for the shape optimization of an electric motor. 

\vspace{5mm}\textbf{Acknowledgement.}
The authors gratefully acknowledge the Austrian Science Fund (FWF) for the financial support of their work via the Doctoral Program DK W1214 (project DK4) on Computational Mathematics. They also thank the Linz Center of Mechatronics (LCM), which is a part of the COMET K2 program of the Austrian Government, for supporting their work 
on the optimization of electrical machines. 
%
%

%
\end{document}